\documentclass[12pt,a4paper]{article}
\usepackage[utf8]{inputenc}
\usepackage[english]{babel}
\usepackage{amsmath}
\usepackage{amsfonts}
\usepackage{amssymb}
\usepackage[left=2cm,right=2cm,top=2cm,bottom=2cm]{geometry}
\usepackage{amsthm}
\usepackage{multicol} 
\usepackage{latexsym}

\newcommand{\bs}[1]{\boldsymbol{#1}}
\newcommand{\Div}[1]{\text{div}\,}
\newtheorem{definition}{{Definition}}

\newtheorem{lemma}{{Lemma}}[section]
\newtheorem{corollary}{Corollary}
\newtheorem{theorem}{{Theorem}}[section]

\numberwithin{equation}{section}
\title{A New Family of Boundary-Domain Integral Equations
for the Dirichlet Problem of the Diffusion Equation in
Inhomogeneous Media with $H^{-1}(\Omega)$ Source Term on Lipschitz Domains}
\author{Z.W. Woldemicheal, C.Fresneda-Portillo\footnote{Corresponding author, c.portillo@brookes.ac.uk}}

\begin{document}\maketitle
\begin{center}
The authors gratefully acknowledge the financial support received from the London Mathematical Society Scheme 5, Award 51809, for Collaboration with Developing Countries.
\end{center}

\abstract{The interior Dirichlet boundary value problem for the diffusion equation in non-homogeneous media is reduced to a system of Boundary-Domain Integral Equations (BDIEs) employing the parametrix obtained in \cite{carlos2} different from \cite{mikhailov1}. We further extend the results obtained in \cite{carlos2} for the mixed problem in a smooth domain with $L^{2}(\Omega)$ right hand side to Lipschitz domains and PDE right-hand in the Sobolev space $H^{-1}(\Omega)$, where neither the classical nor the canonical co-normal derivatives are well defined. Equivalence between the system of BDIEs and the original BVP is proved along with their solvability and solution uniqueness in appropriate Sobolev spaces. }

\section{Introduction}
The popularity of the boundary integral equation method (BIE) is owed to the reduction of dimension of a boundary value problem (BVP) with constant coefficients and homogeneous right hand side defined on a domain of $\mathbb{R}^{n}$. By applying the BIE method, one can reformulate the original BVP in terms of an equivalent integral equation defined exclusively on the boundary of the domain. This method has already been extensively studied for many boundary value problems, for instance: Laplace, Helmholtz, Stokes, Lamé, etc. \cite{hsiao, mclean, steinbach}. This method requires an explicit formula for the fundamental solution of the PDE operator in the BVP which is not always available when the BVP has variable coefficients\cite{mikhailov1,pomp}.

The overcome this issue, one can construct a parametrix (Levi function)\cite[Section 3]{mikhailov1} for the PDE operator and use it to derive an equivalent system of Boundary-Domain Integral Equations following a similar approach as for the BIE method. However, the reduction of dimension no longer applies as volume integrals will appear in the new formulation as a result of the remainder term. This is also the case for non-homogeneous problems with constant coefficients, \cite[Chapter 1 and 2]{hsiao}.

In order to preserve the reduction of dimension, one can use the method of radial integration method (RIM) which allows to transform volume integrals into boundary only integrals\cite{rim2}. This method has been successfully implemented to solve boundary-domain integral equations derived from BVPs with variable coefficients\cite{RIM,2dnumerics}. This method is also able to remove various singularities appearing in the domain integrals.

The recent developments on numerical approximation of the solution of BDIEs show that there are effective and fast algorithms able to compute the solution. For example: the collocation method\cite{Ravnik1, Ravnik2} which, although leads to fully populated matrices, it can be further enhanced by using hierarchical matrix compression and adaptive methods as shown in \cite{numerics} to reduce the computational cost. Localised approaches to reduce the matrix dimension and storage have also been developed\cite{localised, localised2} which lead to sparse matrices. 

Moreover, reformulating the original BVP in the Boundary Domain Integral Equation form can be beneficial, for instance, in inverse problems with variable coefficients\cite{chapko}.

On the one hand, the family of weakly singular parametrices of the form $P^{y}(x,y)$ for the particular operator
\begin{equation}\label{PDE}
\mathcal{A}u(x) := \sum_{i=1}^{3}\dfrac{\partial}{\partial x_{i}}\left(a(x)\dfrac{\partial u(x)}{\partial x_{i}}\right),
\end{equation}
 has been studied extensively studied\cite{mikhailov1,exterior,exterior2, mikhailov18}. Note that the superscript in $P^{y}(x,y)$ means that $P^{y}(x,y)$ is a function of the variable coefficient depending on $y$, this is 
 $$P^{y}(x,y)= P(x,y;a(y))=\dfrac{-1}{4\pi a(y)\vert x-y\vert}.$$ 
 
On the other hand, 
   $$P^{x}(x,y)= P(x,y;a(x))=\dfrac{-1}{4\pi a(x)\vert x-y\vert}.$$
  is another parametrix for the same operator $\mathcal{A}$. In this case, the parametrix depends on the variable coefficient $a(x)$. This parametrix was introduced in \cite{carlos2} for the mixed problem in smooth 3D domains and in \cite{carlosL} for the mixed problem in Lipschitz domains. Some preliminary results for the mixed problem in exterior domains have also been obtained\cite{carlos3}. 
  
  However, most of the numerical methods to solve BDIEs aforementioned are tested for the Dirichlet problem\cite{numerics,Ravnik1,Ravnik2,RIM}. In order to compare the performance of parametrices $P^{x}(x,y)$ and $P^{y}(x,y)$, one needs first to prove the equivalence between the original Dirichlet BVP and the system of BDIEs as well as the uniqueness of solution (well-posedness) of the system of BDIEs what corresponds to the main purpose of this paper.
  
The study of new families of parametrices is helpful at the time of constructing parametrices for systems of PDEs as shown in \cite[Section 1]{carlos2} for the Stokes system. In this case, the fundamental solution for the pressure does not present any relationship with the viscosity coefficient whereas the parametrix for the pressure depends on two variable viscosity coefficients: one depending on $y$ and another depending on $x$, see also \cite{carlos1}. 
   
The parametrix preserves a strong relation with the fundamental solution of the corresponding PDE with constant coefficients. Using this relation, it is possible to establish further relations between the surface and volume potential type operators for the variable-coefficient case with their counterparts for the constant coefficient case, see, e.g. \cite[Formulae (3.10)-(3.13)]{mikhailov1}, \cite[Formulae (4.6)-(4.11)]{carlos1}. 

   Different families of parametrices lead to different relations with their counterparts for the constant coefficient case. For the parametrices considered in this paper, these relations are rather simple, which makes it possible to obtain the mapping properties of the integral potentials in Sobolev spaces and prove the equivalence between the BDIE system and the BVP. After studying the Fredholm properties of the matrix operator which defines the system, its invertibility is proved, what implies the uniqueness of solution of the BDIE system.

In this paper, we extend the results obtained in \cite{carlos3, carlos2, dufera} by considering the source term of the equation  $\mathcal{A}u=f$ in the Sobolev space $ H^{-1}(\Omega)$. This happens for example, when the source term $f$ is the Dirac's delta distribution. The Dirac's delta is an example of distribution that does not belong to the space $L^{2}$ but belongs to $H^{-1}$ and is used in many applications in Physics, Engineering and other mathematical problems\cite{ponce, beyer, trujillo}. This generalisation for the source term introduces an additional issue on the definition of the co-normal derivative which is needed to derive BDIEs.  

The co-normal derivative operator is usually defined with the help of first Green identity, since the function derivatives do not generally exist on the boundary in the trace sense. However this definition is related to an extension of the PDE and its right-hand side from the domain $\Omega,$ where they are prescribed, to the boundary of the domain, where they are not. Since the extension is non-unique, the co-normal derivative appears to be non-unique operator, which is also non-linear in $u$ unless a linear relation between $u$ and the PDE right-hand side extension is enforced. This creates some difficulties particularly in formulating the BDIEs. 

To overcome these technical issues, we introduce a subspace of $H^{1}(\Omega)$ which is mapped by the PDE operator into the space $\widetilde{H}^{-\frac{1}{2}}(\Omega)$ for the right hand side\cite{Mikhailov15}. This allows to define an internal co-normal derivative operator, which is unique, linear in $u$ and coincides with the co-normal derivative in the trace sense if the latter does exist. This approach is applied to the formulation and analysis of direct segregated BDIEs equivalent to the stated Dirichlet BVP with a varaible coefficent and right hand side from $\widetilde{H}^{-1}(\Omega).$ 

Last but not the least, we generalise in this paper the results for the two-dimensional case and smooth boundary domains \cite{dufera}.

\section{Partial Differential Operators in  $\widetilde{H}^{-1}(\Omega)$}
Let $\Omega=\Omega^{+}$ be a bounded simply connected open Lipschitz domain and let $\Omega^{-}:=\mathbb{R}^{3}\smallsetminus\bar{\Omega}^{+}$ be the complementary (unbounded) domain. The Lipschitz boundary $\partial\Omega$ is connected and compact.   

We shall consider the partial differential equation 
\begin{equation}\label{ch4operatorA}
\mathcal{A}u(x):=\sum_{i=1}^{3}\dfrac{\partial}{\partial x_{i}}\left(a(x)\dfrac{\partial u(x)}{\partial x_{i}}\right)=f(x),\,\,x\in \Omega,
\end{equation}
where the variable smooth coefficient $a(x)\in \mathcal{C}^{2}(\overline{\Omega})$ is such that 
\begin{equation*}
0<a_{\rm min}\leqslant a(x)\leqslant a_{\rm max}<\infty,\; \forall x\in\overline{\Omega},\,\, a_{\rm min},a_{\rm max}\in \mathbb{R}, 
\end{equation*}
 $u(x)$ is the unknown function and $f$ is a given distribution on $\Omega$. It is easy to see that if $a\equiv 1$ then, the operator $\mathcal{A}$ becomes $\Delta$, the Laplace operator.
  
 In what follows $\mathcal{D}(\Omega):=C^{\infty}_{comp}(\Omega)$ denotes the space of Schwartz test functions, and $\mathcal{D}^{*}(\Omega)$ denotes the space of Schwartz distributions, $H^{s}(\Omega)$,
$H^{s}(\partial\Omega)$ denote the Bessel potential spaces, where $s\in\mathbb{R}$ (see e.g. \cite{mclean, hsiao} for more details). We recall that the spaces $H^{s}$ coincide with the Sobolev-Slobodetski spaces $W^{s}_{2}$ for any non-negative $s$. We denote by $\widetilde{H}^{s}(\Omega)$ 
the subspace of ${H}^{s}(\mathbb{R}^{3})$, $\widetilde{H}^{s}(\Omega):=\{g:g\in
H^{s}(\mathbb{R}^{3}),~\textrm{supp}~g\subset\overline{\Omega}\}.$ The space $H^{s}(\Omega)$ denotes the space of restriction on $\Omega$ of distributions from $H^{s}(\mathbb{R}^{3})$, defined as $H^{s}(\Omega):=\{r_{_{\Omega}}g:g\in
H^{s}(\mathbb{R}^{3})\}$ where $r_{_{\Omega}}$ denotes the restriction operator on $\Omega.$ Let us defined the dual topological spaces $H^{-1}(\Omega):=[\widetilde{H}^1(\Omega)]^{*}$ and $\widetilde{H}^{-1}(\Omega):=[H^1(\Omega)]^{*}$. 

For $u\in H^{1}(\Omega),$ the partial differential operator $\mathcal{A}$ is understood in the sense of 
distributions. Using the usual notation of distribution theory, equation \eqref{ch4operatorA} can be written as
\begin{equation}\label{eq:zc21}
\langle \mathcal{A}u,v\rangle_{\Omega}=\langle f,v\rangle_{\Omega},\quad \forall v\in
\mathcal{D}(\Omega)
\end{equation}

Using the differentiation properties of distributions, one can obtain the following identity
\begin{equation*}
\langle \mathcal{A}u,v\rangle_{\Omega}=- \langle a\nabla u, \nabla v\rangle_{\Omega}, \quad \forall v\in
\mathcal{D}(\Omega).
\end{equation*}
To simplify the notation, we introduce the operator $\mathcal{E}$ defined as follows
$$\mathcal {E}(u,v):=\int_{\Omega}a(x)\nabla u(x)\cdot\nabla v(x)\,\,dx,$$
which allow us to write equation \eqref{eq:zc21} as
\begin{equation}\label{eq:zc2}
\langle \mathcal{A}u,v\rangle_{\Omega}=-\mathcal {E}(u,v) = \langle f,v\rangle_{\Omega},\quad \forall v\in
\mathcal{D}(\Omega).
\end{equation}

Since the set $\mathcal{D}(\Omega)$ is dense in
$\widetilde{H}^1(\Omega)$, formula \eqref{eq:zc2} defines the continuous linear operator 
$\mathcal{A}: H^1(\Omega)\to H^{-1}(\Omega)$, 
where
\begin{equation*}
\langle\mathcal{A}u,v\rangle_{\Omega}:=-{\cal E}(u,v),~\quad \forall u\in H^{1}(\Omega),\,\, v\in
\widetilde{H}^1(\Omega).
\end{equation*}
Let us also define the so-called aggregate operator\cite[Section 3.1]{traces} of $\mathcal{A}$, as {$\mathcal{\check{A}}: H^1(\Omega)\to \widetilde{H}^{-1}(\Omega)$}
\begin{equation}\label{eq:zc3}
\langle\mathcal{\check{A}}u,v\rangle_{\Omega}
:=-\mathcal{E}(u,v), \quad \forall u,v\in H^{1}(\Omega),
\end{equation}
where the bilinear functional {$\mathcal{E}:H^{1}(\Omega) \times H^{1}(\Omega)\longrightarrow \mathbb{R}$} is defined as
\begin{align*}
\mathcal{E}(u,v)&:=-\int_{\mathbb{R}^{3}}\mathring{E}[a\nabla u](x)\cdot\nabla \mathring{E}{v}(x)\,\,dx =\langle \nabla\cdot\mathring{E}[a\nabla u],\mathring{E}{v}\rangle_{\mathbb{R}^{3}}\nonumber,~~
\forall u, v\in H^1(\Omega),
\end{align*}
where $\mathring{E}:H^{1}(\Omega)\to \widetilde{H}^{1}(\Omega)$ denotes the operator of extension of functions, defined in $\Omega,$ by zero outside $\Omega$ in $\mathbb{R}^{3}$. Note that the functional $\check{\mathcal{E}}$ is continuous due to its symmetry and therefore, so does $\mathcal{\check{A}}$. Now, we can provide an explicit definition for the aggregate operator \eqref{eq:zc3}
\begin{equation*}
\mathcal{\check{A}}u:=\nabla\cdot\mathring{E}[a\nabla u].
\end{equation*}
 For any $u\in H^{1}(\Omega),$ the functional $\mathcal{\check{A}}u$ belongs to
$\widetilde{H}^{-1}(\Omega)$ and is an extension of the functional $\mathcal{A}u\in H^{-1}(\Omega)$ whose domain is thus extended from $\widetilde{H}^1(\Omega)$ to the domain
$H^1(\Omega) $.

\section{Traces, co-normal derivatives and Green identities}
From the trace theorem for Lipschitz domains, we know that the trace of a scalar function $w\in H^{s}(\Omega^\pm)$, $s>1/2$,  belongs to the space $H^{s-\frac{1}{2}}(\partial\Omega)$, i.e., $\gamma^{\pm}w\in H^{s-\frac{1}{2}}(\partial\Omega)$.  Moreover, if $\frac{1}{2}<s<\frac{3}{2},$ the corresponding traces operators {${\gamma^{\pm}:=\gamma^{\pm}_{\partial\Omega}:H^{s}(\Omega^{\pm})\longrightarrow H^{s-\frac{1}{2}}(\partial\Omega)}$} are continuous\cite[Lemma 3.6]{costabel}.
 
For $u\in H^{s}(\Omega)$, $s>3/2$, we can define on $\partial\Omega$ the conormal derivative operators, $T^{\pm}$, in the classical sense
\begin{equation*}\label{ch4conormal}
T^{\pm}_{x}u := \sum_{i=1}^{3}a(x)\gamma^{\pm}\left( \dfrac{\partial u}{\partial x_{i}}\right)^{\pm}n_{i}^{\pm}(x),
\end{equation*}
where $n^{+}(x)$ is the exterior unit normal vector directed \textit{outwards} the interior domain $\Omega$ at a point $x\in\partial\Omega$. Respectively, $n^{-}(x)$ is the unit normal vector directed \textit{inwards} the interior domain $\Omega$ at a point $x\in \partial\Omega$. Sometimes, we will also use the notation $T^{\pm}_{x}u$ or $T^{\pm}_{y}u$ to emphasise the differentiation variable. When the variable of differentiation is obvious or is a dummy variable, we will simply use the notation $T^{\pm}u$.  

It is well known that the classical co-normal derivative operator is generally not well defined if $u\in H^{1}(\Omega)$ \cite[Appendix A]{Mikhailov15} \cite{traces, costabel}. Consequently, to correctly define a conormal derivative, one can draw on the first Green identity. This is the case for the generalised co-normal derivative and the canonical co-normal derivatives \cite[Definition 3.1 and 3.6]{traces}. 

\begin{definition}\label{genco}
Let $u\in H^1(\Omega)$ and $\mathcal{A}u=r_{\Omega}\tilde{f}$ in $\Omega$ for some
$\tilde{f}\in\widetilde{H}^{-1}(\Omega)$. Then, the {\em
generalised co-normal derivative} 
${T^+(\tilde{f},u)\in H^{-\frac{1}{2}}(\partial\Omega)}$ is defined as
\begin{equation*}
\langle{T}^+(\tilde{f},u),w\rangle_{_{\partial\Omega}}:=\langle
\tilde{f},\gamma^{-1}w\rangle_{\Omega}+{\cal E}(u,\gamma^{-1}w)=\langle
\tilde{f}-\mathcal{\check{A}}u,\gamma^{-1}w\rangle_{\Omega},~\quad \forall w\in H^{\frac{1}{2}}(\partial\Omega).
\end{equation*}
\end{definition}


If $u,v\in H^{1}(\Omega)$, $u$ satisifying  $\mathcal{A}u=r_{\Omega}\tilde{f}$ in $\Omega$ for some
$\tilde{f}\in\widetilde{H}^{-1}(\Omega),$ then, the first Green identity holds in the following form
\begin{equation}
\langle{T}^+(\tilde{f},u),\gamma^+v\rangle_{_{\partial\Omega}}=\langle
\tilde{f},v\rangle_{\Omega}+{\cal E}(u,v)=\langle
\tilde{f}-\mathcal{\check{A}}u,v\rangle_{\Omega},~\quad \forall v\in H^1(\Omega)\label{eq:zc6}.
\end{equation}

In order to appropriately define the canonical co-normal derivative\cite[Definition 6.5]{Mikhailov13}, we introduce the following space.
\begin{definition}
Let $s\in\mathbb{R}$ and $\mathcal{A}_*:H^{s}(\Omega)\longrightarrow\mathcal{D}^{*}(\Omega)$ be a linear operator. For $t\ge-\frac{1}{2},$
we introduce the space $$H^{s,t}(\Omega;\mathcal{A}_*):=\{g: g\in
H^s(\Omega):\mathcal{A}_*g|_{_{\Omega}}=\tilde{f}_{g}|_{_{\Omega}}, ~\tilde{f}_{g}\in
\widetilde{H}^t(\Omega)\}$$
endowed  with the norm $$\|g\|_{H^{s,t}(\Omega;\mathcal{A}_*)}:=\left( \|g\|^2_{H^s(\Omega)}+\|\tilde{f}_{g}\|^2_{\widetilde{H}^t(\Omega)}\right) ^{\frac{1}{2}}.$$
\end{definition}

In this paper, $\mathcal{A}_{*}$ will refer to either $\mathcal{A}$ or $\Delta$ in the above definition. Also, we remark, that $H^{1,-\frac{1}{2}}(\Omega ;\mathcal{A})=H^{1,-\frac{1}{2}}(\Omega ;\Delta)$ due to  $\mathcal{A}u-a\Delta u=\nabla a\cdot\nabla u\in L_{2}(\Omega)$ for any $u\in H^{1}(\Omega)$. 

\begin{definition}
For $u\in H^{1,-\frac{1}{2}}(\Omega;\mathcal{A}),$ we define the {\em canonical co-normal derivative} $T^{+}u\in H^{-\frac{1}{2}}(\partial\Omega)$ as,
\begin{equation*}
\langle T^{+}u,w\rangle_{_{\partial\Omega}}:=\langle
\mathcal{\tilde{A}}u,\gamma^{-1}w\rangle_{\Omega}+{\cal E}(u,\gamma^{-1}w)=\langle
\mathcal{\tilde{A}}u-\mathcal{\check{A}}u,\gamma^{-1}w\rangle_{\Omega},\,\,\, \forall w\in H^{\frac{1}{2}}(\partial\Omega),
\end{equation*}
where $\mathcal{\tilde{A}}u :=\mathring{E}(\mathcal{A}u)$.
\end{definition}

The canonical co-normal derivatives $T^{+}u$ is independent of (non-unique) choice of the operator $\gamma^{-1}$; independent of the source term $\tilde{f}$, unlike to generalised co-normal derivative defined in \eqref{eq:zc6}; it is linear with respect to $u$ and has the following continuous mapping property: the operator ${T^{+}:H^{t,-\frac{1}{2}}(\Omega;\mathcal{A})\longrightarrow H^{-\frac{1}{2}}(\partial\Omega)}$ when $t\geq-\frac{1}{2}$\cite[Theorem 3.9,]{traces} \cite[Theorem 2.13]{mikhailov18}. 

Furthermore, if $u\in H^{1,-\frac{1}{2}}(\Omega;\mathcal{A})$ and $v\in H^{1}(\Omega)$, then, the first Green identity for the canonical co-normal derivative holds in the following form, (cf. \cite[Theorem  2.13]{mikhailov18}). 
\begin{equation}\label{eq:zcgi}
\langle T^+u,\gamma^+v\rangle_{\partial\Omega} = {\check{\cal E}}(u,v)+\langle \mathcal{\tilde{A}}u, v\rangle_{\Omega}.
\end{equation}

Furthermore, if $u\in H^{1,-\frac{1}{2}}(\Omega;\mathcal{A})$ and $v\in H^{1}(\Omega)$, then, the first Green identity for the canonical co-normal derivative holds in the following form\cite[Theorem  2.13]{mikhailov18} 
\begin{equation*}
\langle T^+u,\gamma^+v\rangle_{\partial\Omega} = {\cal E}(u,v)+\langle \mathcal{\tilde{A}}u, v\rangle_{\Omega}.
\end{equation*}
In the particular case of $u \in H^{1,0}(\Omega;\mathcal{A})$ and $v\in H^{1}(\Omega)$, then, the first Green identity takes the form
\begin{equation}\label{eq:zcgi0}
\langle T^+u,\gamma^+v\rangle_{\partial\Omega} = {\cal E}(u,v)+\langle \mathcal{A}u, v\rangle_{\Omega}.
\end{equation}

To obtain the second Green identity for $u\in H^{1,-\frac{1}{2}}(\Omega;\mathcal{A})$ and $v\in H^{1}(\Omega)$, we use the first Green identity for the canonical co-normal derivative for $u$, i.e. identity \eqref{eq:zcgi} and subtract it from the first Green identity for the generalised co-normal derivative for $v$, this is, swapping $u$ by $v$ in formula \eqref{eq:zc6}. Hence, supposing that $r_{\Omega}\mathcal{A}v =r_{\Omega}\tilde{f}$ with $\tilde{f}\in \widetilde{H}^{-1}(\Omega)$, we obtain the following second Green identity
 \begin{equation}\label{eq:zc8}
 \langle \tilde{f}, u\rangle_{\Omega}-\langle \mathcal{\tilde{A}}u, v \rangle_{\Omega}=\langle T^+(\tilde{f},v),\gamma^+u\rangle_{\partial\Omega}-
 \langle T^+u,\gamma^+v\rangle_{\partial\Omega}.
  \end{equation}
If $u,v\in H^{1,-\frac{1}{2}}(\Omega;\mathcal{A}),$ then we arrive at the familiar form of the second Green identity for the canonical extension and canonical co-normal derivatives 
  \begin{equation*}
  \langle u,\mathcal{\tilde{A}}v\rangle_{\Omega}-\langle v,\mathcal{\tilde{A}}u\rangle_{\Omega}=\langle T^+v,\gamma^+u\rangle_{\partial\Omega}-
 \langle T^+u,\gamma^+v\rangle_{\partial\Omega}.
  \end{equation*}
In the particular case, when $u,v\in H^{1,0}(\Omega;\mathcal{A}),$ the previous identity becomes
    \begin{equation}
  \langle u,\mathcal{A}v\rangle_{\Omega}-\langle v,\mathcal{A}u\rangle_{\Omega}=\langle T^+v,\gamma^+u\rangle_{\partial\Omega}-
 \langle T^+u,\gamma^+v\rangle_{\partial\Omega}.
  \end{equation}
  
\section{Parametrices and remainders}
We aim to derive boundary-domain integral equation systems for the following \textit{Dirichlet} boundary value problem. This is finding $u\in H^{1}(\Omega)$ satisfying
\begin{subequations}\label{ch4BVP}
\begin{align}
\mathcal{A}u&=f,\hspace{1em}\text{in}\hspace{1em}\Omega\label{ch4BVP1},\\
\gamma^{+}u &= \varphi_{0},\hspace{1em}\text{on}\hspace{1em} \partial\Omega,\label{ch4BVP2}
\end{align}
where $\varphi_{0}\in H^{\frac{1}{2}}(\partial\Omega)$ and $f\in H^{-1}(\Omega).$
\end{subequations}
Let us recall that this BVP is uniquely solvable in $H^{1}(\Omega)$\cite[Theorem 4.3]{steinbach}. 

To obtain a system of boundary-domain integral equations for the boundary value problem \eqref{ch4BVP1}-\eqref{ch4BVP2}, we intend to use Boundary Integral Method (BIM) approach\cite{steinbach}. However, this method requires an explicit fundamental solution which is not always available when the PDE differential operator has variable coefficients, as it is the case for the operator $\mathcal{A}$. To overcome this problem, one can introduce a parametrix\cite{mikhailov1, carlos2,pomp}. 
 \begin{definition} A distribution $P(x,y)$ in two variables $x,y \in \mathbb{R}^{3}$ is said to be a parametrix or Levi function for a differential operator $A_{x}$ differentiating with respect to $x$, if the following identity is satisfied
 \begin{equation}\label{parametrixdef}
A_{x}P(x,y) = \delta(x-y)+R(x,y).
 \end{equation}
 where $\delta(.)$ is the Dirac distribution and $R(x,y)$ is remainder.\end{definition}
A parametrix for a given operator $A_{x}$ might not be unique. This is the case, for example, for the operator $\mathcal{A}$. One parametrix\cite{localised, mikhailov1} is given by
\begin{equation*}
P^y(x,y)=\dfrac{1}{a(y)} P_\Delta(x-y),\hspace{1em}x,y \in \mathbb{R}^{3},
\end{equation*} where
\begin{equation*}\label{ch4fundsol}
P_{\Delta}(x-y) = \dfrac{-1}{4\pi \vert x-y\vert}
\end{equation*}
is the fundamental solution of the Laplace equation.
The remainder corresponding to the parametrix $P^{y}$ is given by
\begin{equation*}
\label{ch43.4} R^y(x,y)
=\sum\limits_{i=1}^{3}\frac{1}{a(y)}\, \frac{\partial a(x)}{\partial x_i} \frac{\partial }{\partial x_i}P_\Delta(x-y)
\,,\;\;\;x,y\in {\mathbb R}^3.
\end{equation*}
{\em In this paper}, for the same operator $\mathcal{A}$ defined in \eqref{ch4operatorA}, we will use another parametrix\cite{carlos2} 
\begin{align}\label{ch4Px}
P(x,y):=P^x(x,y)=\dfrac{1}{a(x)} P_\Delta(x-y),\hspace{1em}x,y \in \mathbb{R}^{3},
\end{align}
which leads to the corresponding remainder 
\begin{align}\label{ch4remainder}
R(x,y) =R^x(x,y) &= 
-\sum\limits_{i=1}^{3}\dfrac{\partial}{\partial x_{i}}
\left(\frac{1}{a(x)}\dfrac{\partial a(x)}{\partial x_{i}}P_{\Delta}(x,y)\right)\\
& =
-\sum\limits_{i=1}^{3}\dfrac{\partial}{\partial x_{i}}
\left(\dfrac{\partial \ln a(x)}{\partial x_{i}}P_{\Delta}(x,y)\right),\hspace{0.5em}x,y \in \mathbb{R}^{3}\nonumber.
\end{align}
Note that the both remainders $R_x$ and $R_y$ are weakly singular, i.e., \[
R^x(x,y),\,R^y(x,y)\in \mathcal{O}(\vert x-y\vert^{-2}).\]
 This is due to the smoothness of the variable coefficient $a(\cdot)$.

\section{Integral operators}
\subsection{Volume potentials}
Following the steps of the boundary-domain integral method \cite{mikhailov1}, we will later on replace $u$ by the parametrix \eqref{ch4Px} in the first Green identity \eqref{eq:zcgi}. This will give an integral representation formula of the solution $u$ in terms of surface and volume potential-type integral operators. In this section, we define these surface and volume integral operators and study their mapping properties which will be applied to prove the main results of this paper. 

For the function $\rho$ defined on $\Omega\subset\mathbb{R}^{3},$ e.g., $\rho\in\mathcal{D}(\overline{\Omega})$ the volume potential and the remainder potential operator, corresponding to parametrix \eqref{ch4Px} and remainder \eqref{ch4remainder} are defined as
 \begin{align}
\textbf{P}\rho(y)&:=\langle P(\cdot,y),\rho\rangle_{\mathbb{R}^{3}}=\int_{\mathbb{R}^{3}}P(x,y)\rho(x)\hspace{0.25em}dx,~y\in\mathbb{R}^{3}
\label{Pdef1}\\
\mathcal{P}\rho(y)&:=\langle P(\cdot,y),\rho\rangle_{\Omega}=\int_{\Omega}P(x,y)\rho(x)\hspace{0.25em}dx,~y\in\Omega\label{Pdef2}\\
\mathcal{R}\rho(y)&:=\langle R(\cdot,y),\rho\rangle_{\Omega}=\int_{\Omega}R(x,y)\rho(x)\hspace{0.25em}dx,~y\in\Omega \label{Rdef}
\end{align}
From \eqref{ch4Px} and \eqref{ch4remainder}, we can see that the both, parametrix and remainder, can be written as a function of the fundamental solution of the Laplace operator. Therefore, one can represent the corresponding volume potential for the parametrix and remainder in terms of the analogous volume potential associated with the Laplace operator (operator $\mathcal{A}$ with $a=1$) as shown below
\begin{align}
\textbf{P}\rho&=\textbf{P}_{\triangle}\left(\dfrac{\rho}{a}\right),\label{ch4relP1}\\ 
\mathcal{P}\rho&=\mathcal{P}_{\Delta}\left(\dfrac{\rho}{a}\right),\label{ch4relP}\\
\mathcal{R}\rho&=\nabla\cdot\left[\mathcal{P}_{\Delta}(\rho\,\nabla \ln a)\right]-\mathcal{P}_{\Delta}(\rho\,\Delta \ln a).\label{ch4relR}
\end{align}
For $\rho\in H^{s}(\Omega),s\in\mathbb{R},$ the operator \eqref{Pdef1} is understood as $\textbf{P}\rho=\textbf{P}_{\triangle}\left(\dfrac{\rho}{a}\right),$ where the Newtonian potential operator $\textbf{P}_{\triangle}$ for the Laplace operator $\Delta$ is well defined in terms of the Fourier transform, on any space $H^{s}(\mathbb{R}^{3}).$ For $\rho\in\widetilde{H}^{s}(\Omega),$  and any $s\in\mathbb{R},$ definitions \eqref{Pdef2} and \eqref{Rdef} can be understood as
\begin{equation}\label{PRrel}
\mathcal{P}\rho=r_{\Omega}\textbf{P}_{\triangle}\left(\dfrac{\rho}{a}\right),\;\;\mathcal{R}\rho=r_{\Omega}\left(  \nabla\cdot\left[\textbf{P}_{\Delta}(\rho\,\nabla \ln a)\right]-\textbf{P}_{\Delta}(\rho\,\Delta \ln a)\right). 
\end{equation}
For the case, $\rho\in H^{s}(\Omega),-\frac{1}{2}<s<\frac{1}{2},$ as \eqref{PRrel} with $\rho$ replaced by $\widetilde{E}\rho$ where $\widetilde{E}:H^{s}(\Omega)\longrightarrow\widetilde{H}^{s}(\Omega),-\frac{1}{2}<s<\frac{1}{2}$ is the unique continous extension operator related with the operator $\mathring{E}$ of extension by zero\cite[Theorem 2.16]{traces}. 

The result \cite[Lemma 3.1]{mikhailov18} provides the mapping properties of the operator $\bs{P}_{\Delta}$, which applied to relations \eqref{ch4relP1}-\eqref{PRrel} provides us with the following result. 

 \begin{theorem}\label{ch4thmUR} Let $\Omega$ be a bounded Lipschitz domain in $\mathbb{R}^{3}$. Then, the operators
 \begin{align}
  \mu\textbf{P}&:H^{s}(\mathbb{R}^{3}) \longrightarrow H^{s+2}(\mathbb{R}^{3}),\hspace{0.5em} s\in \mathbb{R},\;\forall\mu\in\mathcal{D}(\mathbb{R}^{3})\\
 \mathcal{P}&:\widetilde{H}^{s}(\Omega) \longrightarrow H^{s+2}(\Omega),\hspace{0.5em} s\in \mathbb{R},\\
 &: H^{s}(\Omega) \longrightarrow H^{s+2}(\Omega),\hspace{0.5em} -\frac{1}{2}<s<\frac{1}{2},\\
 \mathcal{R}&:\widetilde{H}^{s}(\Omega) \longrightarrow H^{s+1}(\Omega),\hspace{0.5em} s\in \mathbb{R},\label{ch4mpvp3}\\
 &: H^{s}(\Omega) \longrightarrow H^{s+1}(\Omega),\hspace{0.5em} -\frac{1}{2}<s<\frac{1}{2}\,'\\
 &: H^{1}(\Omega) \longrightarrow H^{1,0}(\Omega;A),\hspace{0.5em}\,.\label{ch4mpvp5}
 \end{align}
  are bounded. 
 \end{theorem}
 
 Since $\Omega$ is a bounded domain, then, the compact embedding theorem for Sobolev spaces\cite[Chapter 2]{mclean} can be applied to the remainder operators \eqref{ch4mpvp3}-\eqref{ch4mpvp5} of the previous theorem to obtain the following corollary. 
 
  \begin{corollary}\label{ch4corcompact} For $\frac{1}{2}<s<\frac{3}{2},$ the following operators are compact,
 \begin{align*}
 \mathcal{R}&: H^{s}(\Omega) \longrightarrow H^{s}(\Omega),\\
 \gamma^{+}\mathcal{R}&: H^{s}(\Omega) \longrightarrow H^{s-\frac{1}{2}}(\partial\Omega),\\
 T^{+}\mathcal{R}&: H^{s}(\Omega) \longrightarrow H^{s-\frac{3}{2}}(\partial\Omega).
 \end{align*}
 \end{corollary}
 \subsection{Surface potentials}
 The single layer potential operator and the double layer potential operator associated with the Laplace equation $\Delta u = 0$, are defined as
\begin{align*}
V_{\Delta}\rho(y)&:=-\int_{\partial\Omega} P_{\Delta}(x,y)\rho(x)\hspace{0.25em}dS(x),\\
W_{\Delta}\rho(y)&:=-\int_{\partial\Omega} T{+}_{\Delta}P_{\Delta}(x,y)\rho(x)\hspace{0.25em}dS(x).
\end{align*}
where $T^{+}_{\Delta}$ is the normal derivative operator (i.e. $T^{+}$ with $a(x)\equiv 1$ in \eqref{ch4conormal}). 
Similarly, one can defined the corresponding potentials parametrix-based for $y\in\mathbb R^3$ and $y\notin\partial\Omega $, as 
 \begin{align}
 \label{ch4SL}
V\rho(y)&:=-\int_{\partial\Omega} P(x,y)\rho(x)\hspace{0.25em}dS(x),\\
W\rho(y)&:=-\int_{\partial\Omega} T_{x}^{+}P(x,y)\rho(x)\hspace{0.25em}dS(x).\label{ch4DL}
 \end{align}
Due to \eqref{ch4Px} and the fact that 
$$T^{+}P(x,y) = T^{+}\left(\dfrac{1}{a(x)}P_{\Delta}(x,y)\right)= P_{\Delta}(x,y)T^{+}\left(\dfrac{1}{a(x)}\right) + T^{+}_{\Delta}P_{\Delta}(x,y),$$
the operators $V$ and $W$ can be also expressed in terms the surface potentials and operators associated with the Laplace operator,
\begin{align}
V\rho &= V_{\Delta}\left(\dfrac{\rho}{a}\right),\label{ch4relSL}\\
W\rho &= W_{\Delta}\rho -V_{\Delta}\left(\rho\frac{\partial \ln a}{\partial n}\right).\label{ch4relDL}
\end{align}

Since the mapping properties in Sobolev spaces of the single layer potential and double layer potential for the Laplace equation are well known\cite{mikhailov18,costabel}, one can easily derive analogous mapping properties for the operators $V$ and $W$ as a consequence of the relations \eqref{ch4relSL} and \eqref{ch4relDL}, along with  Theorems 3.3-3.7 proved in \cite{mikhailov18}.
These mapping properties are reflected in the following results which are included for completeness of the paper as some are key to prove the main results. 
\begin{theorem}\label{thmlipcont}
Let $\Omega$ be a bounded Lipschitz domain, let $\frac{1}{2}<s<\frac{3}{2}$. Then, the following operators are bounded,
\begin{align*}
\mu V&:H^{s-\frac{3}{2}}(\partial\Omega)\longrightarrow H^{s}(\mathbb{R}^{3}),\;\forall\mu\in\mathcal{D}(\mathbb{R}^{3});\\
\mu W&:H^{s-\frac{1}{2}}(\partial\Omega)\longrightarrow H^{s}(\Omega);\\
\mu r_{\Omega_{-}}W&:H^{s-\frac{1}{2}}(\partial\Omega)\longrightarrow H^{s}(\Omega_{-}),\;\forall\mu\in\mathcal{D}(\mathbb{R}^{3});\\
 r_{\Omega}V&:H^{-\frac{1}{2}}(\partial\Omega)\longrightarrow H^{1,0}(\Omega_{-};A);\\
 \mu r_{\Omega_{-}}V&:H^{-\frac{1}{2}}(\partial\Omega)\longrightarrow H^{1,0}(\Omega_{-};A),\;\forall\mu\in\mathcal{D}(\mathbb{R}^{3});\\
 r_{\Omega}W&:H^{\frac{1}{2}}(\partial\Omega)\longrightarrow H^{1,0}(\Omega_{-};A);
   \end{align*}
 \begin{align*}
 \mu r_{\Omega_{-}}W&:H^{\frac{1}{2}}(\partial\Omega)\longrightarrow H^{1,0}(\Omega_{-};A),\;\forall\mu\in\mathcal{D}(\mathbb{R}^{3});\\
 \gamma^{\pm} V&:H^{s-\frac{3}{2}}(\partial\Omega)\longrightarrow H^{s-\frac{1}{2}}(\partial\Omega);\\
 \gamma^{\pm} W&:H^{s-\frac{1}{2}}(\partial\Omega)\longrightarrow H^{s-\frac{1}{2}}(\partial\Omega);\\
 T^{\pm} V&:H^{s-\frac{3}{2}}(\partial\Omega)\longrightarrow H^{s-\frac{3}{2}}(\partial\Omega);\\
 T^{\pm} W&:H^{s-\frac{1}{2}}(\partial\Omega)\longrightarrow H^{s-\frac{3}{2}}(\partial\Omega).
\end{align*}
\end{theorem}
The following result follows from the relations \eqref{ch4relSL}-\eqref{ch4relDL} and the analogous jump properties for the harmonic surface potentials \cite[Theorem 3.6]{mikhailov18}.
\begin{corollary}\label{corjump}
Let $\partial\Omega$ be a compact Lipschitz boundary. Let$\varphi\in H^{s-\frac{1}{2}}(\partial\Omega)$ and $\psi\in H^{s-\frac{3}{2}}(\partial\Omega)$ with $\frac{1}{2}<s<\frac{3}{2}$. Then, 
\begin{align}
\label{lipcvw1}\gamma^{+}V\psi-\gamma^{-}V\psi=0,\;\;\gamma^{+}W\varphi-\gamma^{-}W\varphi=-\varphi;&\\
\label{lipcvw2}
T^{+}V\psi-T^{-}V\psi=\psi,\;\;T^{+}W\varphi-T^{-}W\varphi=-(\partial_{n}a)\varphi&.
\end{align}
\end{corollary}
The mapping properties in Theorem \ref{thmlipcont} and Corollary \ref{corjump} imply the following result about the mapping properties of the operators related to the traces and co-normal derivatives of the single and double layer parametrix-based operators \eqref{ch4SL} and \eqref{ch4DL}. Alternatively, the proof directly follows from \cite[Theorem 3.3]{mikhailov18}, relations \eqref{ch4SL} and \eqref{ch4DL}, the trace theorem and mapping properties of the conormal derivative.  
\begin{corollary}\label{corcontDV}
Let $\Omega\subset \mathbb{R^{3}}$ with $\partial\Omega$ compact Lipschitz boundary, $\frac{1}{2}<s<\frac{3}{2}.$ Then, the operators
\begin{align*}
\mathcal{V}&:=\gamma^{+}V=\gamma^{-}V:H^{s-\frac{3}{2}}(\partial\Omega)\longrightarrow H^{s-\frac{1}{2}}(\partial\Omega);\\
\mathcal{W}&:=\frac{1}{2}(\gamma^{+}W+\gamma^{-}W):H^{s-\frac{1}{2}}(\partial\Omega)\longrightarrow H^{s-\frac{1}{2}}(\partial\Omega);\\
\mathcal{W^{'}}&:=\frac{1}{2}(T^{+}V+T^{-}V):H^{s-\frac{3}{2}}(\partial\Omega)\longrightarrow H^{s-\frac{3}{2}}(\partial\Omega);\\
\mathcal{L}&:=\frac{1}{2}(T^{+}W+T^{-}W):H^{s-\frac{1}{2}}(\partial\Omega)\longrightarrow H^{s-\frac{3}{2}}(\partial\Omega).
\end{align*}
are bounded.
\end{corollary}
%

%
The operators $\mathcal{V}, \mathcal{W}, \mathcal{W}'$ can be represented as integral operators with parametrix based kernel 
\begin{align}
\mathcal{V}\rho(y)&:=-\int_{\partial\Omega} P(x,y)\rho(x)\hspace{0.25em}dS(x),\,\,\, y\in\partial\Omega,\nonumber \\
\mathcal{W}\rho(y)&:=-\int_{\partial\Omega} T_{x}P(x,y)\rho(x)\hspace{0.25em}dS(x),\,\,\, y\in\partial\Omega,\nonumber\\
\mathcal{W'}\rho(y)&:=-\int_{\partial\Omega} T_{y}P(x,y)\rho(x)\hspace{0.25em}dS(x),\,\,\, y\in\partial\Omega.\nonumber
\end{align}
By Corollary \ref{corcontDV} and relations \eqref{ch4relSL}-\eqref{ch4relDL}, the operators $\mathcal{V}, \mathcal{W}, \mathcal{W'}$ and $\mathcal{L}$ can be expressed in terms the volume and surface potentials and operators associated with the Laplace operator\cite{carlos1}. 
\begin{align}
\mathcal{V}\rho &= \mathcal{V}_{\Delta} \left( \dfrac{\rho}{a}\right),\label{ch4relDVSL}\\
\mathcal{W}\rho &= \mathcal{W}_{\Delta}\rho -\mathcal{V}_{\Delta}\left(\rho\frac{\partial \ln a}{\partial n}\right),\label{ch4relDVDL}\\
\mathcal{W}'\rho &= a \mathcal{W'}_{\Delta}\left(\dfrac{\rho}{a}\right),\label{ch4relTSL} \\
\mathcal{L}\rho &=a\mathcal{L}_{\Delta}\rho - a\mathcal{W^{'}}_{\Delta}\left(\rho\frac{\partial \ln a}{\partial n}\right).
\label{ch4relTDL}
\end{align}
Furthermore, by the Liapunov-Tauber Theorem \cite[Lemma 4.1]{costabel} for Lipschitz domains,  $\mathcal{L}_{\Delta}\rho = T_{\Delta}^{+}W_{\Delta}\rho = T_{\Delta}^{-}W_{\Delta}\rho$.
\section{Integral representation of the solution in terms of the surface and volume potentials}
In this section, we will obtain an integral representation formula for the solution $u$ of the original BVP \eqref{ch4BVP1}-\eqref{ch4BVP2}. These results will be useful to construct a system of BDIEs equivalent to the original Dirichlet BVP. We will follow a similar approach as in \cite{mikhailov18} but using the new parametrix \eqref{ch4Px}. 

\begin{theorem} Let $u\in H^{1}(\Omega)$. Then, 
\begin{itemize}
\item[(i)] The following integral representation formula holds
\begin{equation*}\label{green3}
u+{\cal R}u+W\gamma^+u={\cal P}\check{A}u\quad{\rm
in}\quad\Omega,
\end{equation*}
where ${\cal P}\check{A}u:=-\mathcal{E}(u,P)$.
\item[(ii)] Let $r_{\Omega}\mathcal{A}u =\tilde{f}$ with $\tilde{f}\in\widetilde{H}^{-1}(\Omega)$. Then, the integral representation formula \begin{equation}\label{eq:3GI}
u+{\cal R}u-V T^+(\tilde{f},u)+W\gamma^+u={\cal
P}\tilde{f}\quad \textrm{in}\quad\Omega,
\end{equation}
holds.

\item[(iii)] The trace of $u$,i.e. $\gamma^{+}u$, can be represented in terms of the surface and volume potentials as follows
\begin{equation}\label{ch4G3co}
\frac{1}{2}\gamma^{+}u+\gamma^{+}\mathcal{R}u-\mathcal{V}T^{+}(\tilde{f},u)+\mathcal{W}\gamma^{+}u=\gamma^{+}\mathcal{P}\tilde{f},\hspace{0.5em}{\rm on\ }\partial\Omega.
\end{equation} 
\end{itemize}
\end{theorem}

\begin{proof}
Let us prove item (i). First, we consider the first Green identity \eqref{eq:zcgi} with the roles of $u$ and $v$ interchanged
\begin{equation}\label{L1}
\langle T^{+}v,\gamma^{+}u\rangle_{\partial\Omega} = \mathcal{E}(u,v)+\langle \mathcal{A}u, v\rangle_{\Omega}.
\end{equation}
In order to apply the first Green identity, we needed $u\in H^{1}(\Omega)$ and $v\in H^{1,0}(\Omega;\mathcal{A})$. Let us take $v:=P(x,y)$ as the parametrix. Let us remark that as $P$ is the parametrix, then
\begin{equation}
\langle \mathcal{A}P, u\rangle_{\Omega} = u+\mathcal{R}u.
\end{equation}
For every distribution $u\in H^{1}(\Omega)$, $\langle \mathcal{A}P, u\rangle_{\Omega}\in H^{1}(\Omega)$, and hence in $ L^{2}(\Omega)$ due to the mapping properties of the operator $\mathcal{R}$ given in Theorem \ref{ch4thmUR}. This implies that, as a distribution, $P\in H^{1,0}(\Omega;\mathcal{A})$
Then, the identity \eqref{L1} can now be reformulated in terms of the surface and volume integral operators as
\begin{equation}\label{L2}
W\gamma^{+}u= \mathcal{E}(u,P)+\mathcal{P}\mathcal{A}u.
\end{equation}
Since $\mathcal{E}(u,P)=-\langle \mathcal{A}P, u\rangle_{\Omega}$ and $\mathcal{A}P=\delta +R$ for being $P$ a parametrix, we obtain
\begin{equation}\label{L3}
W\gamma^{+}u= -u-\mathcal{R}u+\mathcal{P}\mathcal{A}u.
\end{equation}
what implies the result. \\
Let us prove now item [(ii)]. Since $r_{\Omega}\mathcal{A}u =\tilde{f}$ with $\tilde{f}\in\widetilde{H}^{-1}(\Omega)$, then, we need to use the generalised second Green identity \eqref{eq:zc8}, again swapping the roles of $u$ and $v$ 
\begin{equation}
 \langle \tilde{f}, v\rangle_{\Omega}-\langle \mathcal{\tilde{A}}v, u \rangle_{\Omega}=\langle T^+(\tilde{f},u),\gamma^+v\rangle_{\partial\Omega}-
 \langle T^+v,\gamma^+u\rangle_{\partial\Omega}.
\end{equation}
As before, make $v=P(x,y)$ to obtain a new representation formula in terms of the parametrix-based surface and volume potentials 
\begin{equation}
\mathcal{P}\tilde{f}-\langle \mathcal{\widetilde{A}}P, u \rangle_{\Omega}=-VT^+(\tilde{f},u)+ W\gamma^+u.
\end{equation}
Taking into account that $\mathcal{\widetilde{A}}P = \delta +R$ and rearranging terms, we obtain 
\begin{equation}\label{L4}
 u +\mathcal{R}u -VT^+(\tilde{f},u)+ W\gamma^+u=\mathcal{P}\tilde{f}. 
\end{equation}
Item (iii) directly follows from item (ii) by taking the trace of \eqref{L4}, keeping in mind the jump property $\gamma^{+ }W\gamma^+u = -\frac{1}{2}\gamma^+u +\mathcal{W}\gamma^+u$ given by Corollary \ref{corjump} and the mapping properties given in Corollary \ref{corcontDV}. 
\end{proof}

To derive the boundary-domain integral equation systems, we will use the integral representation formulas obtained in the previous theorem. However, we will substitute that both the trace and generalised conormal derivatives are independent from $u$. Hence, we will use the distributions $\Psi$ and $\Phi$ in their place as unknowns alongside $u$, and consider the new boundary domain integral equation
\begin{equation}\label{ch4G3ind}
u+\mathcal{R}u-V\Psi+W\Phi=\mathcal{P}\tilde{f},\hspace{0.5em}{\rm in\ }\Omega.
\end{equation}

We will show now that any triple $(u,\Psi, \Phi)$ satisfying the previous relation, solves the PDE \eqref{ch4BVP1}. 

The following two statements are a generalisation of Lemma 9 and Lemma 10 in \cite{carlosL} to the case where $\tilde{f}\in H_{-1}(\Omega)$.

\begin{lemma}\label{Lem:zc1}
Let $u\in H^{1}(\Omega), \Psi \in H^{-\frac{1}{2}}(\partial\Omega), \Phi \in H^{\frac{1}{2}}(\partial\Omega),~\mathrm{and}~\tilde{f}\in \widetilde{H}^{-1}(\Omega)$ satisfy equation \eqref{ch4G3ind}. Then
\begin{itemize}
\item [(i)] $u$ solves $Au=r_{\Omega}\tilde{f}$, in $\Omega$, 
\item [(ii)] $r_{\Omega}V(\Psi -T^{+}(\tilde{f},u))-r_{\Omega}W(\Phi -\gamma^{+}u)=0$, in $\Omega$.
\end{itemize}
\end{lemma}

\begin{proof}
Take equation \eqref{ch4G3ind} and subtract it from the third Green identity \eqref{eq:3GI} applied to $u$, to obtain \eqref{ch4G3ind} to obtain
\begin{equation}\label{ch4lema1.3}
W(\gamma^{+}u-\Phi)-V(T^{+}(\tilde{f},u)-\Psi)=\mathcal{P}(\check{\mathcal{A}}u-\tilde{f}).
\end{equation}
Let us apply relations \eqref{ch4relP}, \eqref{ch4relSL} and \eqref{ch4relDL} to \eqref{ch4lema1.3} 
\begin{align*}
 V_{\Delta}\left(\dfrac{\tilde{f},u)-\Psi}{a}\right)-W_{\Delta}(\gamma^{+}u-\Phi)
+V_{\Delta}\left(\dfrac{\partial\ln a}{\partial n}\,(\gamma^{+}u-\Phi)\right) = \mathcal{P}_{\Delta}\left(\check{\mathcal{A}}u-\tilde{f}\right)
\end{align*}
Then, apply the Laplace operator to both sides to obtain 
\begin{equation}\label{ch4lema1.5}
\check{\mathcal{A}}u-\tilde{f}=0,
\end{equation}
what implies that $r_{\Omega}\check{\mathcal{A}}u =\mathcal{A}u=r_{\Omega}\tilde{f}$ from where it follows item (i). 
Finally, substituting \eqref{ch4lema1.5} into \eqref{ch4lema1.3}, we prove item (ii).
\end{proof}

The following Lemma is a direct consequence of the invertibility of the direct value of the single layer potential for the Laplace equation\cite[Corollary 8.13]{mclean}. A proof of the Lemma is available in \cite{carlosL}.

\begin{lemma}\label{ch4lemma2}
Let $\Psi^{*}\in H^{-\frac{1}{2}}(\partial\Omega)$. 
\begin{equation}\label{ch4lema2i}
V\Psi^{*}(y) = 0,\hspace{2em}y\in\Omega
\end{equation}
then $\Psi^{*}(y) = 0$.
\end{lemma}

\section{BDIE system for the Dirichlet problem}
We aim to obtain a segregated boundary-domain integral equation system for Dirichlet BVP \eqref{ch4BVP}.  Corresponding formulations for the mixed problem for $u\in H^{1,0}(\Omega;\Delta)$ with $f\in L_{2}(\Omega)$ were introduced and analysed in \cite{carlos1,carlos2,mikhailov1}. 
Let $\tilde{f}\in\widetilde{H}^{-1}(\Omega)$ be an extension of $f\in H^{-1}(\Omega)$, (i.e., $f=r_{\Omega}\tilde{f}$), which always exists, see \cite[Lemma 2.15 and Theorem 2.16]{traces}. Let us represent the generalized co-normal derivative and the trace of the function $u$ as $T^{+}(\tilde{f},u)=\psi,\;\;\gamma^{+}u=\varphi_{0},$ and we will regard the new unknown function $\psi\in H^{-\frac{1}{2}}(\partial\Omega)$ as formally segregated of $u.$ Thus we will look for the couple $(u,\psi)\in H^{1}(\Omega)\times H^{-\frac{1}{2}}(\partial\Omega).$

To obtain one of the possible boundary-domain integral equation systems we will use equation \eqref{ch4G3ind} in the domain $\Omega$, and equation \eqref{ch4G3co} on $\partial\Omega.$ Then we obtain the following system (A1) of two equations for two unknown functions,
\begin{subequations}
\begin{align}
u+\mathcal{R}u-V\psi&=F_{0}\hspace{2em}in\hspace{0.5em}\Omega,\label{ch4SM12v}\\
\gamma^{+}\mathcal{R}u-\mathcal{V}\psi&=\gamma^{+}F_{0}-\varphi_{0}\label{ch4SM12g}\hspace{2em}on\hspace{0.5em}\partial\Omega,
\end{align}
\end{subequations}
where
\begin{equation}\label{ch4F0term}
F_{0}=\mathcal{P}\tilde{f}-W\varphi_{0}.
\end{equation}

Note that for $\varphi_{0}\in H^{\frac{1}{2}}(\partial\Omega),$ we have the inclusion $F_{0}\in H^{1}(\Omega)$ if $\tilde{f}\in\widetilde{H}^{-1}(\Omega)$ due to the mapping properties of the surface and volume potentials given in Theorem \ref{ch4thmUR}, Theorem \ref{thmlipcont} and Corollary \ref{corcontDV}. 

The system (A1), given by \eqref{ch4SM12v}-\eqref{ch4SM12g} can be written in matrix notation as 
\begin{equation*}
\mathcal{A}^{1}\mathcal{U}=\mathcal{F}^{1},
\end{equation*}
where $\mathcal{U}$ represents the vector containing the unknowns of the system,
\begin{equation*}
\mathcal{U}=(u,\psi)^{\top}\in H^{1}(\Omega)\times H^{-\frac{1}{2}}(\partial\Omega),
\end{equation*}
the right hand side vector is \[\mathcal{F}^{1}:= [ F_{0}, \gamma^{+}F_{0} - \varphi_{0} ]^{\top}\in H^{1}(\Omega)\times H^{\frac{1}{2}}(\partial\Omega),\]
and the matrix operator $\mathcal{A}^{1}$ is defined by:
\begin{equation*}
   \mathcal{A}^{1}=
  \left[ {\begin{array}{ccc}
   I+\mathcal{R} & -V  \\
   \gamma^{+}\mathcal{R} & -\mathcal{V} 
  \end{array} } \right].
\end{equation*}

We note that the mapping properties of the operators involved in the matrix imply the continuity of the operator
\begin{equation*}
    \mathcal{A}^{1}: H^{1}(\Omega)\times H^{-\frac{1}{2}}(\partial\Omega)\longrightarrow H^{1}(\Omega)\times H^{\frac{1}{2}}(\partial\Omega).
\end{equation*}
Let us prove that the Dirichlet boundary value problem \eqref{ch4BVP} in $\Omega$ is equivalent to the system of the Boundary Domain Integral Equations \eqref{ch4SM12v}-\eqref{ch4SM12g}.
\begin{theorem}\label{ch4EqTh}
Let $\varphi_{0}\in H^{\frac{1}{2}}(\partial\Omega)$, $f\in H^{-1}(\Omega)$ and $\tilde{f}\in\widetilde{H}^{-1}(\Omega)$ is such that $r_{\Omega}\tilde{f}=f$. 
\begin{enumerate}
\item[i)] If a function $u\in H^{1}(\Omega)$ solves the Dirichlet BVP \eqref{ch4BVP}, then the couple $(u, \psi)^{\top}\in H^{1}(\Omega)\times H^{-\frac{1}{2}}(\partial\Omega)$ where
\begin{equation}\label{ch4eqcond}
\psi=T^{+}(\tilde{f},u),\hspace{2em}on\hspace{0.5em}\partial\Omega,
\end{equation}
solves the BDIE system (A1). 

\item[ii)] If a couple $(u, \psi)^{\top}\in H^{1}(\Omega)\times H^{-\frac{1}{2}}(\partial\Omega)$ solves the BDIE system (A1) then $u$ solves the BVP and the functions $\psi$ satisfy \eqref{ch4eqcond}.

\item[iii)] The system (A1) is uniquely solvable. 
\end{enumerate}
\end{theorem}

\begin{proof}
$i)$. Let $u\in H^{1}(\Omega)$ be a solution of the boundary value problem \eqref{ch4BVP}. Then, from Definition \ref{genco}, the generalised conormal derivative is well defined for the pair $(\tilde{f},u)$. Hence, let $\psi:= T^{+}(\tilde{f},u)\in H^{-\frac{1}{2}}(\partial\Omega)$. Replacing the pair $(u,\psi)$ in \eqref{ch4SM12v}-\eqref{ch4SM12g}, we arrive to the third Green identities for $u$ and $\gamma^{+}u$ given in the Lemma \ref{Lem:zc1}. Therefore, the pair $(u,\psi)$ solves the BDIEs \eqref{ch4SM12v}-\eqref{ch4SM12g}.

$ii)$. Let now the couple $(u, \psi)^{\top}\in H^{1}(\Omega)\times H^{-\frac{1}{2}}(\partial\Omega)$ solve the BDIE system. Taking the trace of the equation \eqref{ch4SM12v} and substract it from the equation \eqref{ch4SM12g}, we obtain
\begin{equation}\label{ch4M12a1}
\gamma^{+}u=\varphi_{0}, \hspace{1em} \text{on}\hspace{0.5em}\partial\Omega.
\end{equation}
i.e, $u$ satisfies the Dirichlet condition \eqref{ch4BVP2}. 
Equation \eqref{ch4SM12g} and Lemma \ref{Lem:zc1} with $\Psi=\psi,\Phi=\varphi_{0}$ imply that $u$ is a solution of PDE \eqref{ch4BVP1} and 
\begin{equation}\label{ch4M12a2}
r_{\Omega}V(\psi - T^{+}(\tilde{f},u)) - r_{\Omega}W(\varphi_{0}-\gamma^{+}u)=0.
\end{equation}
Since $\gamma^{+}u=\varphi_{0}$, then $\gamma^{+}u-\varphi_{0}=0$. Hence, equation \ref{ch4M12a2} becomes
\begin{equation}\label{ch4M12a22}
V(\psi - T^{+}(\tilde{f},u))=0 \text{ in } \Omega.
\end{equation}
Applying now Lemma \ref{ch4lemma2} with $\Psi^{*} =\psi - T^{+}(\tilde{f},u)$ then implies $\Psi^{*}=0$. This implies that $\psi = T^{+}(\tilde{f},u)$. Thus, $u$ obtained from the solution of BDIE system (A1) solves the Dirichlet problem.

Item $iii)$ immediately follows from the equivalence between the BDIE system and the BVP. Since the Dirichlet boundary value problem \eqref{ch4BVP1}-\eqref{ch4BVP2} is uniquely solvable\cite[Theorem 4.3]{steinbach}, so it is the BDIE system \eqref{ch4SM12v}-\eqref{ch4SM12g}.\end{proof}

Let us now prove the invertibility of the operator $\mathcal{A}^{1}$
\begin{theorem}
The operator \[\mathcal{A}^{1}:H^{1}(\Omega)\times H^{-\frac{1}{2}}(\partial\Omega)\longrightarrow H^{1}(\Omega)\times H^{\frac{1}{2}}(\partial\Omega),\] is invertible.
\end{theorem}

\begin{proof}
To prove the invertibility, let $\mathcal{A}_{0}^{1}$ be the matrix operator defined by
\begin{equation*}\label{ch4M012}
  \mathcal{A}_{0}^{1}:=
  \left[ {\begin{array}{ccc}
   I & -V  \\
   0 & -\mathcal{V} \\
  \end{array} } \right].
\end{equation*}
As a result of compactness properties of the operators $\mathcal{R}$ and $\gamma^{+}\mathcal{R}$ (cf. Corollary \ref{ch4corcompact}), the operator $\mathcal{A}_{0}^{1}$ is a compact perturbation of operator $\mathcal{A}^{1}$. The operator  $\mathcal{A}_{0}^{1}$ is an upper triangular matrix operator and invertibility of its diagonal operators $I:H^{1}(\Omega)\longrightarrow H^{1}(\Omega)$ and $\mathcal{V}:H^{-\frac{1}{2}}(\partial\Omega)\longrightarrow H^{\frac{1}{2}}(\partial\Omega)$ (cf. Theorem \cite[Section 4]{carlosL}). This implies that 
\begin{center}
$\mathcal{A}_{0}^{1}: H^{1}(\Omega)\times H^{-\frac{1}{2}}(\partial\Omega)\longrightarrow H^{1}(\Omega)\times H^{\frac{1}{2}}(\partial\Omega)$
\end{center}
is an invertible operator. Thus $\mathcal{A}^{1}$ is a Fredholm operator with zero index. Hence the Fredholm property and the injectivity of the operator $\mathcal{A}^{1}$, provided by item $iii)$ of Theorem \ref{ch4EqTh}, imply the invertibility of operator $\mathcal{A}^{1}$.
\end{proof}

\section{Conclusions}
A new parametrix for the diffusion equation in non homogeneous media with Lipschitz domain and source term in $H^{-1}(\Omega),$ allows us to obtain an equivalent and uniquely solvable system of Boundary-Domain Integral Equations.  

Hence, further investigation about the numerical advantages of using one family of parametrices over another will follow. Now, numerical methods can be applied when the source term belongs to $H^{-1}(\Omega)$\cite{trujillo,ponce,beyer}. 

Further work will consist of extending the results presented in this paper to unbounded domains, non-smooth coefficients or other BVP problems with different boundary conditions as well as providing a localised version of the BDIEs (A1) presented in this paper, inspired by the works, \cite{mikhailov1,localised}.  

We highlight again that analysing BDIEs for different parametrices, i.e. depending on the variable coefficient $a(x)$ or $a(y),$ is crucial to understand the analysis of BDIEs derived with parametrices that depend on the variable coefficient $a(x)$ and $a(y)$ at the same time, as it is the case for the Stokes system \cite{carlos2,carlos1}.





\end{document}